\documentclass[12pt]{amsart}
\usepackage{amsfonts,amsthm,latexsym,amsmath,amssymb,amscd,amsmath, mathrsfs, url, cite, graphicx }
\usepackage[pdftex]{color}
\usepackage[bookmarks=true,hyperindex,pdftex,colorlinks, citecolor=blue,linkcolor=blue, urlcolor=blue]{hyperref}
 
\newcounter{count}
\numberwithin{count}{section}
\newtheorem{Lemma}[count]{Lemma}
\newtheorem{Remark}[count]{Remark}
\newtheorem{Example}[count]{Example}
\newtheorem{Definition}[count]{Definition}
\newtheorem{theorem}[count]{Theorem}
\newtheorem{Corollary}[count]{Corollary}
\newtheorem{Theorem}[count]{Theorem}
\newtheorem{Conjecture}[count]{Conjecture}
\newtheorem{Problem}[count]{Open Problem}
\newtheorem{Statement}[count]{Statement}

\begin{document}

\author[A.~Vishnyakova]{Anna Vishnyakova}

\address{Department of Mathematics, Holon Institute of Technology,
Israel}
\email{annalyticity@gmail.com}

\title[Polynomials interpolated totally positive sequences]
{Polynomials interpolated totally positive sequences}
\maketitle

\begin{abstract}
A real sequence  $(a_k)_{k=0}^\infty$ is called {\it totally positive}  
if all minors of the infinite Toeplitz matrix $ \left\| a_{j-i}  
\right\|_{i, j =0}^\infty$ are nonnegative (where $a_k=0$
for $k<0$). In this paper, we investigate the following question:
for which real polynomials $P$ the sequence $(P(k))_{k=0}^\infty$
is totally positive? We establish a few new necessary conditions, 
sufficient conditions, present a number of  important examples and formulate
several open problems.

\end{abstract}



\section{Introduction}

\begin{Definition}
 A sequence of nonnegative numbers $(a_k)_{k=0}^\infty$
is called  totally positive, if all minors of the
infinite matrix

\begin {equation}
\label{mat}
 \left(
  \begin{array}{ccccc}
   a_0 & a_1 & a_2 & a_3 &\ldots \\
   0   & a_0 & a_1 & a_2 &\ldots \\
   0   &  0  & a_0 & a_1 &\ldots \\
   0   &  0  &  0  & a_0 &\ldots \\
   \vdots&\vdots&\vdots&\vdots&\ddots
  \end{array}
 \right)
\end {equation}
 are non-negative. 
\end{Definition}

The class of totally positive sequences is denoted by $\mathrm{TP}.$
The class of generating functions of totally positive sequences
$f(x)=\sum_{k=0}^\infty a_k x^k$
is  denoted by  $\widetilde{\mathrm{TP}}.$

Concept of total positivity found numerous applications and was
studied from many different sides (see, for example, T.~Ando, \cite{ando}, 
S.~Karlin,  \cite{tp},  or A.~Pinkus, \cite{pin}).  It has applications  
in distribution of zeros of polynomials and entire functions,  P\'olya 
frequency sequences, unimodality and log-concavity,  stochastic 
processes and approximation theory, mechanical systems,  planar  
networks, combinatorics and representation theory.

Totally  positive  sequences were introduced
by M.~Fekete in 1912 (see~\cite{fek}) in connection with the problem 
of determining the exact number of positive zeros of a real polynomial. 
$\mathrm{TP}$  sequences arise in many areas of mathematics and its
applications.

Note that checking that all minors of an infinite matrix are non-negative is 
often a very difficult task. The work \cite{kv} contains an easily verifiable 
sufficient condition.

The class $\widetilde{\mathrm{TP}}$ was completely described in the
classical theorem by M.~Aissen, A.~Edrei, I.J.~Schoenberg, A.~Whitney.

{\bf  Theorem A} (M.~Aissen, A.~Edrei, I.J.~Schoenberg, A.~Whitney, \cite{aissen}).
{ \it  Let $(a_k)_{k=0}^\infty$  be a given sequence of nonnegative numbers.  
Then $f(z)=\sum_{k=0}^\infty a_kz^k \in \widetilde{\mathrm{TP}}$ if and
only if
$$f(z)=C z^n e^{\gamma z}\prod_{k=1}^\infty (1+\alpha_kz)/(1-\beta_kz),$$
where $C\ge 0, n\in \mathbb{ N}\cup\{0\},
\gamma\ge0,\alpha_k\ge0,\beta_k\ge0,\sum(\alpha_k+\beta_k)
<\infty.$}

The following fact is a simple corollary of Theorem A.

\begin{Corollary}  A polynomial with nonnegative coefficients 
$P(x) =\sum _{k=0}^n a_k x^k $   has only real zeros if and only 
if the sequence of its coefficients is totally positive: $(a_0,
a_1, \ldots , a_n, 0, 0, \ldots ) \in \mathrm{TP}.$
\end{Corollary}

The main problem we are interested in is the following open problem.
\begin{Problem}
\label{pr1}
To describe the set of real polynomials $P\in \mathbb{R}[x]$ such that
the sequence $(P(k))_{k=0}^\infty$ is totally positive.
\end{Problem}

We consider the following identity
$$ \sum_{k=0}^\infty x^k = \frac{1}{1-x}.  $$ 
For $m\in \mathbb{N}\cup \{0\}$  let's differentiate this formula 
$m$ times and multiply by $\frac{x^m}{m!}.$ We obtain
\begin{equation}
\label{f1}
\sum_{k=0}^\infty \frac{k(k-1)\cdot \ldots \cdot (k-m+1)}{m!} x^k
=\frac{x^m}{(1-x)^{m+1}}.
\end{equation}

Let $P\in \mathbb{R}[x]$ be a real polynomial of degree $n\in \mathbb{N}.$ 
Let us decompose this polynomial in the form
\begin{eqnarray}
\label{f2}
&
P(x)= a_0 + a_1 x +a_2 \frac{x(x-1)}{2!} +a_3 \frac{x(x-1)(x-2)}{3!}\\
\nonumber &
+  \ldots +a_n \frac{x(x-1)\cdot \ldots \cdot (x-n+1)}{n!}.
\end{eqnarray}
From (\ref{f1}) we get
\begin{eqnarray}
\label{f3}
&
\sum_{k=0}^\infty P(k) x^k= \frac{a_0}{1-x} +  \frac{a_1 x}{(1-x)^2} 
+ \frac{a_2 x^2}{(1-x)^3} + \ldots + \frac{a_n x^n}{(1-x)^{n+1}}\\
\nonumber &
=\frac{1}{1-x} \left(a_0 + a_1 \cdot \frac{x}{1-x} + a_2 \cdot \frac{x^2}{(1-x)^2} +
\ldots +  a_n \cdot \frac{x^n}{(1-x)^n} \right) \\
\nonumber &
=\frac{1}{(1-x)^{n+1}} \left(a_0 (1-x)^n + a_1 x(1-x)^{n-1} + a_2 x^2(1-x)^{n-2} +
\ldots +  a_nx^n\right).
\end{eqnarray}
Using Theorem A, we conclude that $(P(k))_{k=0}^\infty \in \mathrm{TP}$ if and only if all zeros of the
polynomial 
$$S(x)=a_0 (1-x)^n + a_1 x(1-x)^{n-1} + a_2 x^2(1-x)^{n-2} +
\ldots +  a_n x^n$$ 
are real and non-positive and  all coefficients are non-negative.

Denote by 
$$Q(x) =a_0 +a_1 x +a_2 x^2 + \ldots + a_n x^n.$$
We have
\begin{equation}
\label{ff1}
S(x) =(1-x)^n Q\left(\frac{x}{1-x}\right).
\end{equation}
Thus, all zeros of the polynomial $S$ are real and non-positive if and only if all zeros of the 
polynomial $Q$ are real and belong to the segment $[-1, 0]$ (we note that the possible zero of $Q$ 
at the point $-1$ does not correspond to any zero of $S$).
In connection with this observation we define the following linear operator.

\begin{Definition} By $\mathcal{L}$ we denote the linear operator  $\mathcal{L} : \mathbb{R}[x] \rightarrow 
\mathbb{R}[x]$ such that
\begin{eqnarray}
\label{f4}
&  \mathcal{L}(a_0 + a_1 x +a_2 \frac{x(x-1)}{2!} 
+  \ldots +a_n \frac{x(x-1)\cdot \ldots \cdot (x-n+1)}{n!})
\\  \nonumber &
= a_0 +a_1 x +a_2 x^2 + \ldots + a_n x^n.
\end{eqnarray}
\end{Definition}
As we have shown, the Open problem~\ref{pr1}  is equivalent to the following Open problem.
\begin{Problem}
\label{pr2}
To describe the set of real polynomials $P\in \mathbb{R}[x]$ with positive leading
coefficients such that  the polynomials $\mathcal{L}(P)$ have only real zeros located 
in the segment $[-1, 0].$
\end{Problem}

\begin{Remark}
\label{r1}
Suppose that a polynomial 
$$P(x)= a_0 + a_1 x +a_2 \frac{x(x-1)}{2!} 
+  \ldots +a_n \frac{x(x-1)\cdot \ldots \cdot (x-n+1)}{n!}$$
is such that
$$\mathcal{L}(P)(x)
= a_0 +a_1 x +a_2 x^2 + \ldots + a_n x^n$$ 
has only real zeros located in the segment $[-1, 0].$ Then all
coefficients of $\mathcal{L}(P)$ are of the same signs. If $a_n >0,$
then there exists $l=0, 1, 2, \ldots, n$ such that $a_j=0$ for all
$j < l,$ and $a_j >0$ for $j=l, l+1, \ldots, n.$

\end{Remark}

The following theorem by F.~Brenti gives the simple sufficient condition for a polynomial
to interpolate a $\mathrm{TP}$-sequence.

{\bf Theorem B} (F.\,Brenti, \cite{BrentiMemoir}, see also \cite{BrentiContMath}, c.f. \cite[Lemma 1]{kav}). {\it  Let 
$P\in \mathbb{R}[x]$ be a polynomial with only real zeros, and let $\lambda(P), \Lambda(P)$ 
be the smallest and the largest zeros of $P.$ Suppose that $P(x)=0$ for all $x\in \left(\left[\lambda(P), -1\right] \cup \left[0, \Lambda(P)\right]\right) \cap \mathbb{Z}.$  Then the polynomial $\mathcal{L}(P)$ has only real zeros 
located in the segment $[-1, 0].$ }

Later we will study how close the sufficient conditions in the previous Theorem are to the necessary ones.

Next Theorem by D.G.~Wagner shows that the set of polynomials that interpolate $\mathrm{TP}$-sequences
is closed under multiplication.

{\bf Theorem C} (D.G.\,Wagner, \cite{wagner}). { \it  Let 
$P_1, P_2 \in \mathbb{R}[x]$ be real polynomials such that both $\mathcal{L}(P_1)$ and 
$\mathcal{L}(P_2)$ have only real zeros located in the segment $[-1, 0].$ Then the
polynomial $\mathcal{L}(P_1\cdot P_2)$ also has only real zeros 
located in the segment $[-1, 0].$ }

\begin{Remark}   It is known that the set of $\mathrm{TP}$-sequences is not closed under
term-by-term multiplication. The following example was given in \cite{kavish}.
We consider the sequence $ (k)_{k=0}^\infty$ with the generating
function $\sum_{k=0}^\infty  k z^k= \frac{z}{(1-z)^2}.$ By Theorem~A, $ (k)_{k=0}^\infty
\in \mathrm{TP}.$ Let us consider the function $f(z) = \frac{1}{(1-z)(2-z)}= \sum_{k=0}^\infty
b_k z^k$  (we have  $b_k =1 - \frac{1}{2^{k+1}}$). By Theorem~A, $ (b_k)_{k=0}^\infty
\in \mathrm{TP}.$ But $\sum_{k=0}^\infty k b_k z^k = z f^\prime(z) =
 \frac{z(3 - 2z)}{(1-z)^2 (2-z)^2},$ this function has a positive zero, so the sequence
 of its coefficients is not a $\mathrm{TP}$-sequence. For further details about term-by-term 
multiplication of totally positive sequences  see \cite{kavish}.
 
 By Theorem~C, the set of totally positive 
 sequences interpolated by polynomials is closed under
term-by-term multiplication.
\end{Remark}

The following open problems connected with the Problem~\ref{pr2} are of interest.

\begin{Problem}
\label{pr3}
To describe the set of real polynomials $P\in \mathbb{R}[x]$ with positive leading
coefficients and only real non-positive zeros such that  the polynomials $\mathcal{L}(P)$
have only real zeros located in the segment $[-1, 0].$
\end{Problem}

\begin{Problem}
\label{pr3}
To describe the set of real polynomials $P\in \mathbb{R}[x]$ with positive leading
coefficients and only real non-negative zeros such that  the polynomials $\mathcal{L}(P)$
have only real zeros located in the segment $[-1, 0].$
\end{Problem}

\begin{Problem}
\label{pr3a}
To describe the set of real polynomials $P\in \mathbb{R}[x]$ with positive leading
coefficients such that  the polynomials $\mathcal{L}(P)$
have only real negative zeros.
\end{Problem}

\section{Some useful formulas}

1. We will start with the famous Newton formulas for finite differences.
Suppose that $P(x)= a_0 + a_1 x +a_2 \frac{x(x-1)}{2!} 
+  \ldots +a_n \frac{x(x-1)\cdot \ldots \cdot (x-n+1)}{n!}.$ We consider
the following linear operator
\begin{equation}
\label{fd}
\Delta:  \mathbb{R}[x] \rightarrow \mathbb{R}[x], \quad (\Delta(P))(x) = P(x+1) -P(x).
\end{equation}
Then the following well-known formulas are valid
\begin{eqnarray}
\label{fd1}
&   a_0 = P(0); \\ \nonumber &
a_k = \Delta^k(P)(0) = \sum_{j=0}^k (-1)^j   C_k^j P(k-j),  k=1, 2, \ldots, n.
\end{eqnarray}

2.  For every $n \in \mathbb{N}$ the following identity is valid 
\begin{equation}
\label{f6}
\sum_{k=0}^n C_n^k \   \frac{x(x-1)\cdot \ldots \cdot (x-k+1)}{k!}= \frac{(x+1)(x+2) \cdot 
\ldots \cdot (x+n)}{n!}
\end{equation}
and can be easily proved by induction.

3. Let $P(x)= a_0 + a_1 x +a_2 \frac{x(x-1)}{2!} +  \ldots +a_n \frac{x(x-1)\cdot \ldots \cdot (x-n+1)}{n!}$ 
and $Q(x)= \mathcal{L}(P)(x) = a_0 +a_1 x +a_2 x^2 + \ldots + a_n x^n.$ We have
$$
  Q(-(x+1)) = \sum_{k=0}^n (-1)^k a_k \sum_{j=0}^k C_k^j x^j =: \tilde{Q} (x), 
$$
thus
$$ \mathcal{L}^{-1}(\tilde{Q} )(x) = \sum_{k=0}^n (-1)^k a_k \sum_{j=0}^k C_k^j  \frac{x(x-1)\cdot \ldots \cdot (x-j+1)}{j!}.  $$
Using (\ref{f6}), we get
$$\sum_{k=0}^n  (-1)^k a_k \sum_{j=0}^k C_k^j  \frac{x(x-1)\cdot \ldots \cdot (x-j+1)}{j!} = $$
$$ \sum_{k=0}^n (-1)^k  \frac{a_k}{k!} (x+1)(x+2)\cdot \ldots \cdot (x +k ) = $$
$$   \sum_{k=0}^n   \frac{a_k}{k!} (-x-1)(-x-2)\cdot \ldots \cdot (-x -k ) = P(-x-1).$$
Finally we obtain
\begin{equation}
\label{fd2}
\mathcal{L}(P(-x-1)) (x) = Q(-x-1).
\end{equation}
Note that if $Q$ has only real zeros located in the segment $[-1, 0],$ 
then $Q(-x-1)$ also  has only real zeros located in the segment $[-1, 0].$
Thus, we have proved the following statement.
 
\begin{Statement}
\label{St0}
Let $P$ be a real polynomial of degree $n\in \mathbb{N}.$  If $(P(k))_{k=0}^\infty \in \mathrm{TP},$ then $\left((-1)^n 
P(-k-1)\right)_{k=0}^\infty \in \mathrm{TP}.$
\end{Statement}

\section{Necessary conditions}

Suppose that a polynomial 
$$P(x)= a_0 + a_1 x +a_2 \frac{x(x-1)}{2!} 
+  \ldots +a_n \frac{x(x-1)\cdot \ldots \cdot (x-n+1)}{n!}$$
is such that
$$\mathcal{L}(P)
= a_0 +a_1 x +a_2 x^2 + \ldots + a_n x^n$$ 
has only real zeros located in the segment $[-1, 0].$ We have mentioned  the 
necessary condition that all  coefficients of $\mathcal{L}(P)$ are of the same signs. 

The following statement gives some more simple necessary conditions.

\begin{Statement}
\label{st1}
Suppose that a real polynomial 
$$P(x)= a_0 + a_1 x +a_2 \frac{x(x-1)}{2!} 
+  \ldots +a_n \frac{x(x-1)\cdot \ldots \cdot (x-n+1)}{n!},$$
$a_n=1, $   is such that $(P(k))_{k=0}^\infty \in \mathrm{TP}.$ Then

1. All real zeros of $P$ belong to the segment $[- n, n-1].$

2. For every $j=0, 1, 2, \ldots, n-1$ we have $a_j \leq C_n^j.$

3. For every $j= 1, 2, \ldots, n-1$ such that $a_{j-1}>0$  we have
$$\frac{a_j^2}{a_{j-1}a_{j+1}} \geq \frac{j+1}{j} \cdot \frac{n-j+1}{n-j}$$
(Newton inequalities).

\end{Statement}

\begin{proof}

The statement that all non-negative zeros of $P$ are in the segment $[0, n-1]$
follows from the fact that all coefficients of $P$ are of the same signs.
The statement that all non-positive zeros of $P$ are in the segment $[-n, 0]$
follows from the previous statement and formula~(\ref{fd2}). Statements~2 and 3
are well-known necessary conditions for a polynomial to have all real zeros
in the statement $[-1, 0].$
\end{proof}

It is interesting to compare the following necessary conditions with Theorem~B.

\begin{Theorem}
\label{th1}
Suppose that a real polynomial 
$$P(x)= a_0 + a_1 x +a_2 \frac{x(x-1)}{2!} 
+  \ldots +a_n \frac{x(x-1)\cdot \ldots \cdot (x-n+1)}{n!},$$
$a_n=1, $   is such that $(P(k))_{k=0}^\infty \in \mathrm{TP}.$ 

1. If there exists $k=1, 2, \ldots, n-1$ such that $P(k)=0,$ then
$P(0) = P(1) = \ldots = P(k-1)=0.$

2.  For every $m\in \mathbb{N}$ we get
\begin{equation}
\label{f110}
P(- m) = \frac{1}{(m-1)!}  \left(x^{m-1} Q(x)\right)^{(m-1)}|_{x=-1}.
\end{equation}

3.  If there exists $l \in \mathbb{N}$ such that $P(-l)=0,$ then
$P(-1) = P(-2) = \ldots = P(-(l-1))=0.$

4. If $\deg P = 2l+1, l \in \mathbb{N}\cup \{0\},$  then there is a root of
$P$ in the segment $[-1, 0].$

\end{Theorem}

\begin{proof}

1. We have for all $k=1, 2, \ldots, n-1$
\begin{equation}
\label{f7}
P(k) = \sum_{j=0}^k C_k^j  a_j.
\end{equation}
If k $P(k)=0,$ then, since all coefficients $a_j$ are of the same signs, we obtain
$a_0=a_1= \ldots = a_k =0$, thus $P(0) = P(1) = \ldots = P(k-1)=0.$

2. For  $P(x)= a_0 + a_1 x +a_2 \frac{x(x-1)}{2!} 
+  \ldots +a_n \frac{x(x-1)\cdot \ldots \cdot (x-n+1)}{n!},$ 
we denote by  $Q(x) = \mathcal{L}(P) (x) = a_0 +a_1x +a_2 x^2 + \ldots + a_n x^n.$
We have
\begin{equation}
\label{f8a}
P(0) = a_0  = Q(0).
\end{equation}
\begin{equation}
\label{f8}
P(-1) = a_0 - a_1 + a_2 -a_3 + \ldots + (-1)^n a_n = Q(-1).
\end{equation}
\begin{equation}
\label{f9}
P(-2) = a_0 - 2 a_1 + 3a_2 - \ldots + (-1)^n (n+1) a_n = (x Q(x))^\prime|_{x=-1}.
\end{equation}
Since $(P(k))_{k=0}^\infty \in \mathrm{TP},$ all zeros of $Q$ are real and in the segment $[-1, 0].$
Thus, all zeros of the polynomial $xQ(x)$ are real and in the segment $[-1, 0].$ Suppose 
that $P(-2)=0.$  That means that  $(x Q(x))^\prime|_{x=-1} =0,$  whence  $-1$ is a root of $Q$ 
of multiplicity $\geq 2.$ In particular, $Q(-1) = P(-1) =0.$

For every $m\in \mathbb{N}$ we get
$$
P(- m) = \frac{1}{(m-1)!} \left( \frac{(m-1)! }{0!} a_0 - \frac{m!}{1!} a_1 + \frac{(m+1)!}{2!} a_2 -  \frac{(m+2)!}{3! } a_3  \right.$$   
$$\left.   + \ldots + (-1)^n \frac{(m+n-1)!}{n!} a_n\right ) =
\frac{1}{(m-1)!}  \left(x^{m-1} Q(x)\right)^{(m-1)}|_{x=-1}.
$$

3.  Since $(P(k))_{k=0}^\infty \in \mathrm{TP},$ all zeros of $Q$ are real and in the segment $[-1, 0].$
Thus, all zeros of the polynomial $x^{l-1} Q(x)$ are real and in the segment $[-1, 0].$ Suppose 
that $P(- l)=0.$  That means that  $\left(x^{l-1} Q(x)\right)^{(l-1)}|_{x=-1} =0,$  whence  $-1$ 
is a root of $Q$ of multiplicity $\geq l.$ Using (\ref{f110}), we obtain that  
$P(-1) = P(-2) = \ldots = P(-(l-1))=0.$

4. Let  $\deg P = \deg Q = 2l+1, l \in \mathbb{N}\cup \{0\}. $ Since $(P(k))_{k=0}^\infty \in \mathrm{TP},$ 
all zeros of $Q$  are real and in the segment $[-1, 0].$  Thus, $Q(0)\cdot Q(-1) \leq 0.$
By (\ref{f8a}) and (\ref{f8}) we get $P(0)\cdot P(-1) \leq 0.$ Hence, then there is a 
root of $P$ in the segment $[-1, 0].$


\end{proof}

In fact, we proved the following statement  about the possible 
roots of a polynomial $Q$ at the endpoints of a segment $[-1, 0].$

\begin{Statement}
\label{st2}
Suppose that a real polynomial 
$$P(x)= a_0 + a_1 x +a_2 \frac{x(x-1)}{2!} 
+  \ldots +a_n \frac{x(x-1)\cdot \ldots \cdot (x-n+1)}{n!},$$
$a_n=1, $   is such that $(P(k))_{k=0}^\infty \in \mathrm{TP},$  and
$Q(x)= a_0 + a_1 x +a_2 x^2 + \ldots + a_n x^n .$

1. The following four statements are equivalent:

1a. The polynomial $Q$ has a zero of
multiplicity $m\in \mathbb{N}$ at the point $x=0.$ 

1b. $a_0 = a_1 = \ldots =a_{m-1}=0.$

1c. $P(m-1)=0.$

1d. $P(0) = P(1) = \ldots = P(m-1)=0.$

2. The following three statements are equivalent.

2a. The polynomial $Q$ has a zero of
multiplicity $m\in \mathbb{N}$ at the point $x=-1.$

2b. $P(-m)=0.$

1c. $P(-m) = P(-(m-1)) = \ldots = P(-1)=0.$

\end{Statement}

\section{Sufficient conditions}

Let  $Q(x) = \sum_{k=0}^n a_k x^k$  be a polynomial with positive coefficients. 
We define the second quotients  $q_n$ as follows:

\begin{equation}
\label{qqq} 
q_k=q_k(Q):=\frac {a_{k-1}^2}{a_{k-2}a_k},\
2 \leq k \leq n.
\end{equation}

The formulas below follow by iteratively applying the definition of the second quotients.

\begin{equation}
\label{defq}
 a_k=\frac
{a_1}{q_2^{k-1} q_3^{k-2} \ldots q_{k-1}^2 q_k} \left(
\frac{a_1}{a_0} \right) ^{k-1},\    2 \leq k \leq n.
\end{equation}

The second quotients $q_k$ provide useful information about the distribution of zeros 
of the associated polynomial. This connection dates back to classical work of 
J.I. Hutchinson~\cite{hut}. In 1926, Hutchinson established one of the first general 
coefficient-based criteria for a polynomial (or an entire function)  with positive 
coefficients  to have only real non-positive zeros.  

{\bf Theorem D} (J. I. ~Hutchinson, \cite{hut}). { \it Let $Q(x)=
\sum_{k=0}^n a_k x^k$, $a_k > 0$ for all $k$. 
If
$$q_k(Q)\geq 4 \ \ \mbox{ for all} \ \ 2 \leq k \leq n,$$  
then all zeros of $Q$ are real and negative.}

It is easy to show that, if the estimation of $q_k(Q)$ only from below is
given then the constant  $4$  in $q_k (Q) \geq 4$ is the smallest possible to conclude
that $Q$ has only real zeros. However, if we have the estimation of $q_k(Q)$ from below and  
from above, then the constant $4$ in the condition $q_k(Q) \geq 4$ can be reduced to 
conclude that all zeros of a polynomial  $Q$  are real and negative, see  \cite{HishAn}.

{\bf Theorem E} (T.H.~ Nguyen, A.~Vishnyakova,  \cite{HishAn}).
{ \it  Let $Q(x) = \sum_{k=0}^n a_k x^k,$ $a_k > 0,$  be a polynomial, and $n \geq 4.$ Suppose that 
there exists $\alpha,  1 + \sqrt{5} \leq \alpha < 4,$ such that
 $q_k(Q) \in \left[\alpha, \frac{8}{\alpha(4 - \alpha)} \right]$ for all $k =2, 3, \ldots, n.$  
Then  all zeros of $Q$ are real and negative. }

The following theorem is a corollary of Theorems~D and E.

\begin{theorem}
\label{th2}
Let $P$ be a real polynomial 
$$P(x)= a_0 + a_1 x +a_2 \frac{x(x-1)}{2!} 
+  \ldots +a_n \frac{x(x-1)\cdot \ldots \cdot (x-n+1)}{n!},$$
$a_j>0 $ for all $j=0, 1, \ldots, n.$

1. Suppose that $P$ satisfies the conditions
$$ \frac {a_{k-1}^2}{a_{k-2}a_k} \geq 4, \    2 \leq k \leq n, $$
and
$$ \frac{a_{n-1}}{a_n} \leq 1. $$
Then $(P(k))_{k=0}^\infty \in \mathrm{TP}.$

2. Suppose that $n \geq 4$ and there exists $\alpha,  1 + \sqrt{5} \leq \alpha < 4,$ such that
 $\frac {a_{k-1}^2}{a_{k-2}a_k} \in \left[\alpha, \frac{8}{\alpha(4 - \alpha)} \right]$ for all 
 $k =2, 3, \ldots, n,$  and $ \frac{a_{n-1}}{a_n} \leq 1. $  Then $(P(k))_{k=0}^\infty \in \mathrm{TP}.$

\end{theorem}

\begin{proof}
Theorems D and E show that under assumptions of Theorem~\ref{th2} 
all zeros of  $Q(x) = \mathcal{L}(P) (x) = a_0 +a_1x +a_2 x^2 + \ldots + 
a_n x^n$ are real and negative. It remains to prove that these zeros 
are in the segment$[-1, 0].$ Under assumptions of Theorem~\ref{th2} 
we have 
$$0 < \frac{a_0}{a_1} < \frac{a_1}{a_2} <  \ldots < \frac{a_{n-2}}{a_{n-1}} < \frac{a_{n-1}}{a_n},  $$
thus, for $x\in  (-\infty, - \frac{a_{n-1}}{a_n}]$ we get
$$ a_0 < a_1 |x| < a_2 |x|^2 < \ldots < a_{n-1}|x|^{n-1} < a_n |x|^n.  $$
So, for all $x\in  (-\infty, - \frac{a_{n-1}}{a_n}]$ we obtain
$$(-1)^n Q(x) = (a_n |x|^n - a_{n-1}|x|^{n-1}) + (a_n |x|^{n-2} - a_{n-1}|x|^{n-3}) +\ldots >0.  $$
Whence $Q$ does not have zeros in $ (-\infty, - \frac{a_{n-1}}{a_n}],$ and since $ \frac{a_{n-1}}{a_n} \leq 1$
all zeros of $Q$ are in the segment $[-1, 0].$ 
\end{proof}

\section{Some examples}

We will start with the following trivial example.

\begin{Example}
\label{ex1} For a polynomial $P$ of degree $1$ we have $\mathcal{L}(P) = P,$
so the polynomial $\mathcal{L}(P)$ has only real zero located in the segment $[-1, 0]$ if
and only if the polynomial $P$ has only real zero located in the segment $[-1, 0].$
\end{Example}

\begin{Remark} Polynomial $P_1(x)=x+2$ has only real negative zeros, but the sequence
$(P_1(k))_{k=0}^\infty $ is not a $\mathrm{TP}$-sequence. Using Laguerre's Theorem on
$CZDS$-sequences  (see \cite{gpol} or \cite[pp. 314-321]{pol}) and Theorem by  
G.~P\'olya and J.~Schur with the description  of the multiplier sequences  (see \cite{polsch}, 
\cite[pp. 100-124]{pol} and\cite[pp. 29-47]{O}), one can prove that
if a polynomial $P$ has only real negative zeros, then the sequence $(\frac{P(k)}{k!})_{k=0}^\infty $ 
is totally positive.

\end{Remark}

\begin{Example}
\label{ex2}    Let $ P(x) = (x+\alpha)(x+\beta), \alpha, \beta \in \mathbb{R}.$ Then 
it is easy to calculate that  
$\sum_{k=0}^\infty (k+\alpha)(k+\beta) x^k = 
\frac{(1-\alpha)(1-\beta) x^2 + (1+ \alpha + \beta - 
2 \alpha \beta) x+ \alpha \beta}{(1-x)^3}. $ Thus,
$$(P(k))_{k=0}^\infty 
\in \mathrm{TP} \Leftrightarrow 
\left\{ 
  \begin{array}{c}
   (1-\alpha)(1-\beta)\geq 0 \\
   1+ \alpha + \beta - 2 \alpha \beta \geq 0 \\
  \alpha \beta \geq 0 \\
  D= (1+ \alpha + \beta - 2 \alpha \beta)^2 -
  4 \alpha \beta  (1-\alpha)(1-\beta) \geq 0. \\
    \end{array}
 \right.
 $$
 The following picture shows the corresponding set of points $(\alpha, \beta).$
By  (\ref{fd2}) and Statement~\ref{St0}, it is symmetric
 with respect to the point $(\frac{1}{2}, \frac{1}{2}).$

 \begin{figure}[h!]
\begin{center}
\includegraphics[width=10cm]{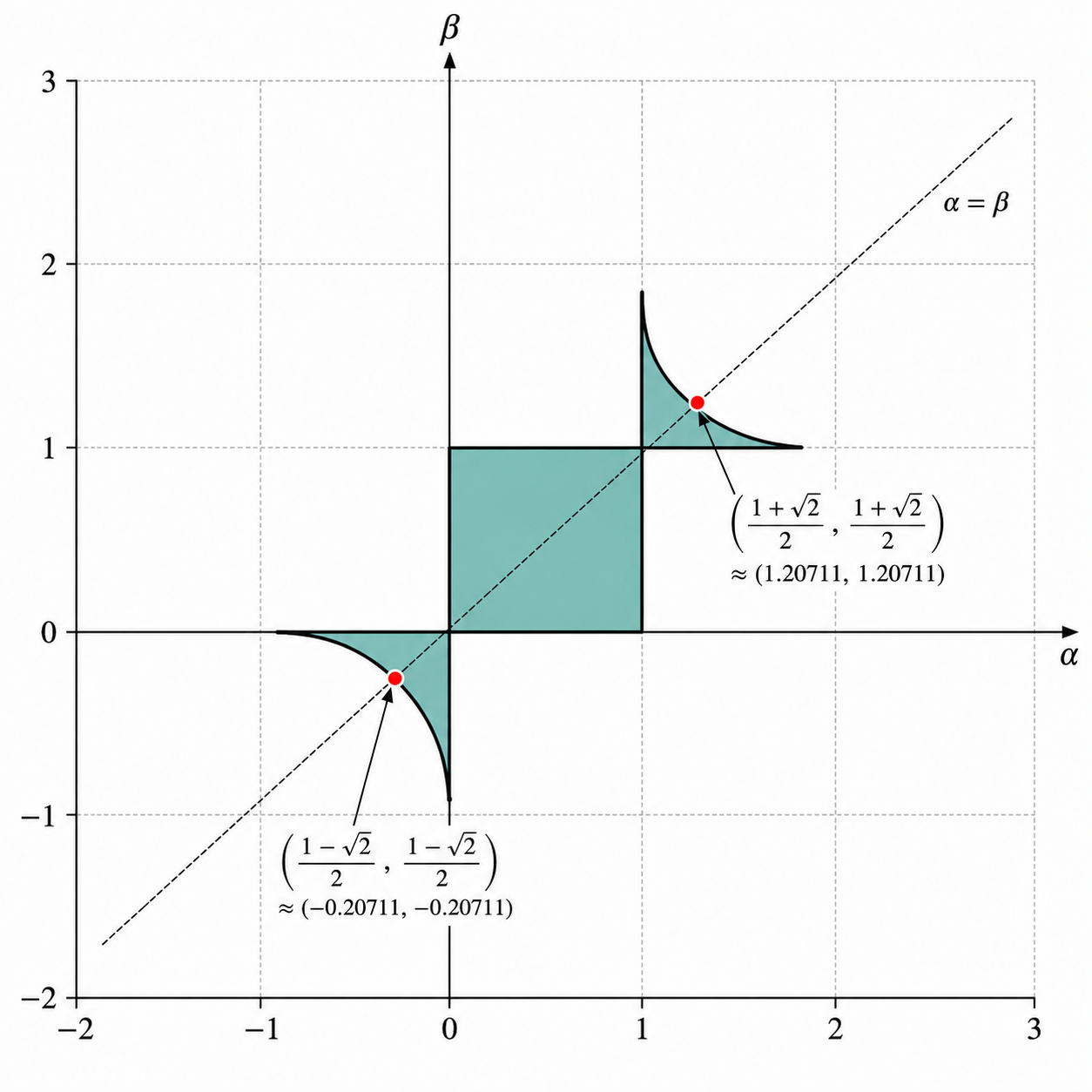}
\end{center}
\end{figure}
\end{Example}

This example shows that sufficient conditions in Theorem~B
are not necessary ones.  For example, let 
$$P_1(x)= \left(x+ \frac{\sqrt{2} +1}{2}\right)^2 = 2\cdot \frac{x(x-1)}{2}+ (\sqrt{2} +2)x + 
\frac{3+2\sqrt{2}}{4},$$
then
$$ Q_1(x)= \mathcal{L}(P_1)(x) = 2x^2 +  (\sqrt{2} +2)x + 
\frac{3+2\sqrt{2}}{4}= 2\left(x + \frac{2+\sqrt{2}}{4}  \right)^2. $$
The polynomial $ Q_1$  has positive coefficients and only real zeros located in the segment $[-1, 0].$
So, $(P_1(k))_{k=0}^\infty \in \mathrm{TP}.$ We see that the polynomial  $P_1$ has zeros with modulus greater 
than 1, but $P_1(-1) \ne 0.$ 

Let
$$P_2(x)= \left(x - \frac{\sqrt{2} -1}{2}\right)^2 = 2\cdot \frac{x(x-1)}{2}+(2- \sqrt{2}) x + 
\frac{3-2\sqrt{2}}{4},$$
then
$$ Q_2(x)= \mathcal{L}(P_2)(x) = 2x^2 +  (2 -\sqrt{2})x + 
\frac{3-2\sqrt{2}}{4}= 2\left(x + \frac{2-\sqrt{2}}{4}  \right)^2. $$
The polynomial $ Q_2$  has positive coefficients and only real zeros located in the segment $[-1, 0].$
So, $(P_2(k))_{k=0}^\infty \in \mathrm{TP}.$ We see that the polynomial  $P_2$ has positive zeros, 
but $P_2(0) \ne 0.$ 

\begin{Example}
\label{ex2a}  
Let $P(x)= x^2 + \frac{1}{8} = 2\cdot \frac{x(x-1)}{2}+ x + \frac{1}{8},$
then
$$ Q(x)= \mathcal{L}(P)(x) = 2x^2 + x + \frac{1}{8}=2\left(x+\frac{1}{4}\right)^2. $$
The polynomial $ Q$  has positive coefficients and only real zeros located in the segment $[-1, 0].$
So, $(P(k))_{k=0}^\infty \in \mathrm{TP}.$ We see that zeros of the polynomial  $P$ are not real.
\end{Example}

\begin{Example}
\label{ex2b} 
For $P_n(x) =x^n, n\in \mathbb{N},$ we obtain
$$ \sum_{k=0}^\infty k^n x^k = \frac{S_n(x)}{(1-x)^{n+1}},   $$
where $S_n(x) =\sum_{k=1}^{n}A(n,k)x^k$ is the Euler polynomial generated 
by the Euler's numbers $A(n,k)$ counting the number of permutations 
of  $\{1,2,\ldots, n\}$  with exactly $k-1$ descents, i.e. pairs $\sigma(i)>\sigma(i+1)$.
Using (\ref{ff1}) we get
$$Q_n(x)= \mathcal{L}(P_n)(x) =   (1+x)^n S_n\left(\frac{x}{1+x}\right) =
\sum_{k=1}^{n}A(n,k)x^k(1+x)^{n-k}.$$

Thus, our linear operator $\mathcal{L}$ in the standard  base of the space $\mathbb{R}[x]$
can be written in the form
$$\mathcal{L}(x^n) = \sum_{k=1}^{n}A(n,k)x^k(1+x)^{n-k}.$$
Some important properties of the linear operator $\mathcal{L}$ in the  form
written above  (more precisely, linear operator $\mathcal{A}$ such that $\mathcal{A}(x^n) 
= \sum_{k=1}^{n}A(n,k)x^k$) were studied in  the papers \cite{branden} by P.~Br\"and\'en
and K.~Jochemko, and \cite{athan} by C.A.~Athanasiadis.

\end{Example} 

\begin{Example}
\label{ex2c} 
Let $P_n(x) =\left(x + \frac{1}{2}\right)^n =:
\sum_{k=0}^n a_k   \frac{x(x-1)\cdot \ldots \cdot (x-k+1)}{k!}, 
n\in \mathbb{N}.$ Using (\ref{fd1}) we obtain
$$ a_k=  \sum_{j=0}^k (-1)^j   C_k^j P_n(k-j) = \sum_{j=0}^k (-1)^j   C_k^j 
\left(\frac{2k-2j+1}{2}  \right)^n,  k= 0, 1,  \ldots, n.
 $$
 Then
 $$ Q_n(x)= \mathcal{L}(P_n)(x) = \sum_{k=0}^n x^k  \sum_{j=0}^k (-1)^j   C_k^j 
\left(\frac{2k-2j+1}{2}  \right)^n.$$
Using Theorem~B by F.~Brenti, one can conclude that all zeros of $Q_n$ are
real and located in the segment $[0, 1].$ These polynomials are a shifted and scaled 
variation of the Eulerian polynomials of type B, typically denoted as $B_n(z)$ (also
 known as MacMahon's Eulerian polynomials). 

The classical Type B Eulerian polynomials can be defined by the generating series
$$\sum_{m=0}^\infty (2m+1)^n z^m = \frac{B_n(z)}{(1-z)^{n+1}}.$$

The family $Q_n(x)$ is directly related to $B_n(z)$ via the following exact rational transformation
$$Q_n(x) = \frac{(1+x)^n}{2^n} B_n\left(\frac{x}{1+x}\right).$$
Because of the direct relationship to $B_n(z)$, which is known to have distinct, real, and strictly negative 
roots, we can say that the roots of $Q_n(x)$  lie strictly in the open interval $(-1, 0)$, they are distinct
and symmetric around $-1/2.$

\end{Example} 

In connection with the previous example (see also Open Problem~\ref{pr4})
we consider the following set of polynomials.

\begin{Example}
\label{ex2d} 
Let $P_{n, t}(x) =\left(x + t\right)^n =:
\sum_{k=0}^n a_k(t)   \frac{x(x-1)\cdot \ldots \cdot (x-k+1)}{k!}, 
n\in \mathbb{N}.$ Using (\ref{fd1}) we obtain
$$ a_k(t)=  \sum_{j=0}^k (-1)^j   C_k^j P_{n, t}(k-j) = \sum_{j=0}^k (-1)^j   C_k^j 
(k-j+t)^n,  k= 0, 1,  \ldots, n.
 $$
 Then
 $$ Q_{n, t}(x)= \mathcal{L}(P_n)(x) = \sum_{k=0}^n x^k  \sum_{j=0}^k (-1)^j   C_k^j 
(k-j+t)^n.$$

For a given $n\in \mathbb{N}$ we want to find the set of values of $t$ such that 
$(P_{n, t}(k))_{k=0}^\infty \in \mathrm{TP}, $ in other words such that all zeros of $ Q_{n, t}$
are real and located in the segment $[-1, 0].$ We denote this set by 
$$M_n = \{ t\in \mathbb{R} | (P_{n, t}(k))_{k=0}^\infty \in \mathrm{TP}   \}.$$
By (\ref{fd2}), if $(P_{n, t}(k))_{k=0}^\infty \in \mathrm{TP},$ then $((-1)^n P_{n, t}(-1-k))_{k=0}^\infty \in TP,$
and $(-1)^n P_{n, t}(-1-x) =(x+1-t)^n =P_{n, 1-t}.$ Thus, $t\in M_n \Rightarrow (1-t)\in M_n,$
so it is sufficient to consider the case $t \geq \frac{1}{2}.$

The answer is simple for odd  values of $n.$  If $n$ is an odd number, $Q_{n, t}(0)=P_{n, t}(0)= t^n >0$  and 
all zeros of $ Q_{n, t}$ are real and located in the segment $[-1, 0], $ then  $Q_{n, t}(-1)\leq 0.$
It is easy to check that $Q_{n, t}(-1)  = P_{n, t}(-1) =(-1+t)^n \leq 0$  (see (\ref{f8})). So, $t\leq 1.$
We obtain that $t\in M_n \Rightarrow t \in [0, 1].$ On the other hand, using Theorem~B by F.~Brenti 
we get $[0, 1] \subset M_n.$ Whence, for every $k\in \mathbb{N}\cup\{0\}$ we have $M_{2k+1} = [0, 1].$

For the case of even $n$ the situation is much more complicated. We mention that by statement 3 in Theorem~\ref{th1} we get that for all $k, j\in \mathbb{N}, j \geq 2,$
we have $-j \notin M_{2k}.$ Calculation shows that the following conjecture is true.

\begin{Conjecture} For all  $k\in \mathbb{N}$ we have $M_{2k} = [-c_{2k}, 1+ c_{2k}], 0 < c_{2k} <2,$ and $c_2 \leq c_4 \leq c_6 \leq \ldots$ 
\end{Conjecture}

The proof of this fact and the possible value of the limit $\lim_{k\to \infty} c_{2k} $ are open questions. 

\end{Example} 

\begin{Example}
\label{ex3}
 For $n \in \mathbb{N}$ consider a polynomial 
\begin{equation}
\label{f5}
P_{0,n}(x) = \frac{x(x-1)\cdot \ldots \cdot (x-n+1)}{n!}.
\end{equation}
We have $\mathcal{L}(P_{0,n}) = x^n,$ so the polynomial $\mathcal{L}(P_{0,n})$ has only real zeros located 
in the segment $[-1, 0].$ Zeros of $P_{0,n}$ are $\{0, 1, 2,  \ldots, n-1\},$ so these zeros satisfy
sufficient conditions of Theorem~B.

For a given $n \in \mathbb{N}$ and $s=0, 1, 2, \ldots, n$ we consider a polynomial
\begin{equation}
\label{f66}
P_{s,n}(x) = \frac{(x+s)(x+s-1)\cdot \ldots \cdot (x+s-n+1)}{n!} = P_{0,n}(x+s).
\end{equation}
We have
$$P_{s,n}(x) = \frac{(x+s)(x+s-1)\cdot \ldots \cdot (x+s-n+1)}{n!} =$$
 $$ \frac{(x+s-1)\cdot \ldots \cdot (x+s-n+1)(x+s-n)}{n!}  + $$
 $$ \frac{(x+s-1)(x+s-2)\cdot \ldots \cdot (x+s-n+1)}{(n-1)!} =   
 P_{s-1,n}(x) + P_{s,n-1}(x).$$
 The following formula can be easily proved by induction
\begin{equation}
\label{f67}
\mathcal{L} (P_{s,n})(x) = x^{n-s}(1+x)^s, n \in \mathbb{N},\  s=0, 1, 2, \ldots, n.
\end{equation}
For $s=n$ one can see that $P_{n,n}(x) =  \frac{(x+1)(x+2) \cdot 
\ldots \cdot (x+n)}{n!} $ and  $\mathcal{L} (P_{n,n})(x) = (1+x)^n $
(cf. with (\ref{f6})).

Zeros of $P_{s,n}$ are $\{ -s, -(s-1), \ldots, -(s-(n-2)), -(s-(n-1))\},$ so these zeros satisfy
sufficient conditions of Theorem~B.

\end{Example}

\begin{Example}
\label{ex4}
For a given $n \in \mathbb{N}\cup\{0\}$ consider a polynomial 
\begin{equation}
\label{f68}
Q_n(x)= n! \left(x+ \frac{1}{2}\right)^n = n! \sum_{k=0}^n C_n^k \frac{x^k}{2^{n-k}}.
\end{equation}
Formula (\ref{fd2}) shows that if the set of zeros of a polynomial $Q_n = \mathcal{L}(P_n)$ is symmetric 
with  respect to the point $\frac{1}{2},$ then  the set of zeros of a polynomial $P_n$ is also symmetric with
respect to the point $\frac{1}{2}.$ The set of polynomials  $Q_n(x)= n! \left(x+ \frac{1}{2}\right)^n$ is interesting 
for us because for every $n$ the polynomial $Q_n$ has zeros that are in some sense ``the most symmetric'' with  
respect to  the point $\frac{1}{2}.$

All zeros of $Q_n$ are real and belong to the segment $[-1, 0].$
Let us denote by
\begin{equation}
\label{f69}
P_n(x) = \mathcal{L}^{-1} \left(Q_n\right)(x) = n! \sum_{k=0}^n C_n^k \frac{x(x-1)(x-2)
\cdot \ldots \cdot (x-k+1)}{k! 2^{n-k}}.
\end{equation}
Note that $Q_n(-1-x) = (-1)^n Q_n(x),$ thus by (\ref{fd2})
\begin{equation}
\label{f70}
P_n(-1-x) = (-1)^n P_n(x).
\end{equation}
Denote by
\begin{equation}
\label{f71}
S_n(x) = P_n\left(- \frac{1}{2} +x\right),
\end{equation}
so that
\begin{equation}
\label{f72}
S_n(- x) =  (-1)^n S_n(x).
\end{equation}
Calculations show that
$$S_0(x)=1,$$
$$S_1(x) =x,$$
$$S_2(x)= x^2 +\frac{1}{4},$$
$$S_3(x)= x^3  + \frac{5x}{4},$$
$$S_4(x)= x^4 + \frac{7 x^2}{2} + \frac{9}{16}.$$
It is easy to check that all zeros of polynomials $S_k, 0\leq k \leq 4,$
are purely imaginary, so all zeros of polynomials $P_k, 0\leq k \leq 4,$
have real parts equal $- \frac{1}{2}.$

\begin{theorem}
\label{th1.0} For every $n\in \mathbb{N}\cup\{0\}$ all zeros of the polynomial
$P_n(x) = n! \sum_{k=0}^n C_n^k \frac{x(x-1)(x-2)
\cdot \ldots \cdot (x-k+1)}{k! 2^{n-k}}$ are simple and have real parts equal 
$- \frac{1}{2}.$ Equivalently, all zeros of the polynomial
$S_n(x) = P_n\left(- \frac{1}{2} +x\right)$ are simple and purely imaginary.
Moreover, zeros of polynomials $P_n$ and $P_{n+1}$ interlace for all
$n\in \mathbb{N}.$ Equivalently,  zeros of polynomials $S_n$ and 
$S_{n+1}$ interlace for all  $n\in \mathbb{N}.$
\end{theorem}

\begin{proof}
By (\ref{f69}) and  (\ref{f71}) we have
\begin{equation}
\label{f73}
S_n(x) =  n! \sum_{k=0}^n \frac{C_n^k }{k! 2^{n-k}} \prod_{j=1}^k \left(x- \frac{2j-1}{2}\right).
\end{equation}

To prove Theorem~\ref{th1.0} we need the following lemma.
\begin{Lemma}
\label{L1}
For all $n \geq 2$  the following recurrence relation holds
\begin{equation}
\label{f74}
S_n(x) =  x S_{n-1}(x) + \frac{(n-1)^2}{4} S_{n-2}(x).
\end{equation}
\end{Lemma}
\begin{proof}
We have
$$ x S_{n-1}(x) + \frac{(n-1)^2}{4} S_{n-2}(x) = x(n-1)! \sum_{k=0}^{n-1} \frac{C_{n-1}^k }{k! 2^{n-1-k}}
 \prod_{j=1}^k \left(x- \frac{2j-1}{2}\right) $$
$$  +\frac{(n-1)}{4} (n-1)!   \sum_{k=0}^{n-2} \frac{C_{n-2}^k }{k! 2^{n-2-k}}
 \prod_{j=1}^k \left(x- \frac{2j-1}{2}\right)  $$
$$=    (n-1)! \sum_{k=0}^{n-1} \frac{C_{n-1}^k }{k! 2^{n-1-k}}
 \prod_{j=1}^{k} \left(x- \frac{2j-1}{2}\right)\cdot \left(x- \frac{2k+1}{2}\right)  $$   
$$+  (n-1)! \sum_{k=0}^{n-1} \frac{C_{n-1}^k }{k! 2^{n-1-k}} 
 \prod_{j=1}^k \left(x- \frac{2j-1}{2}\right) \cdot \frac{2k+1}{2} $$
 $$ +\frac{(n-1)}{4} (n-1)!   \sum_{k=0}^{n-2} \frac{C_{n-2}^k }{k! 2^{n-2-k}}
 \prod_{j=1}^k \left(x- \frac{2j-1}{2}\right)  $$
 $$=    (n-1)! \sum_{k=0}^{n-1} \frac{C_{n-1}^k }{k! 2^{n-1-k}}
 \prod_{j=1}^{k+1} \left(x- \frac{2j-1}{2}\right) $$   
$$+  (n-1)! \sum_{k=0}^{n-1} \frac{C_{n-1}^k }{k! 2^{n-1-k}} \cdot \frac{2k+1}{2} 
 \prod_{j=1}^k \left(x- \frac{2j-1}{2}\right) $$
 $$ +\frac{(n-1)}{4} (n-1)!   \sum_{k=0}^{n-2} \frac{C_{n-2}^k }{k! 2^{n-2-k}}
 \prod_{j=1}^k \left(x- \frac{2j-1}{2}\right)  $$
$$ = \left(\frac{(n-1)!}{2^n} +\frac{(n-1)!(n-1)}{2^n}  \right)  $$
$$+  \sum_{k=1}^{n-1} 
\prod_{j=1}^k \left(x- \frac{2j-1}{2}\right) \cdot \left(\frac{(n-1)! C_{n-1}^{k-1}}{2^{n-k} (k-1)!} + 
\frac{(n-1)! C_{n-1}^{k}}{2^{n-k-1} k!}\cdot \frac{2k+1}{2} \right.$$
$$\left.  +\frac{(n-1)}{4} \cdot 
\frac{(n-1)! C_{n-2}^k}{2^{n-k-2} k!} \right)  + \prod_{j=1}^n \left(x- \frac{2j-1}{2}\right)\cdot 
\frac{(n-1)!}{2^0 (n-1)!}. $$
\begin{equation}
\label{f75}
=\frac{n!}{2^n}+ \sum_{k=1}^{n-1}  b_k
\prod_{j=1}^k \left(x- \frac{2j-1}{2}\right) +  \prod_{j=1}^n \left(x- \frac{2j-1}{2}\right),
\end{equation}
where
$$b_k = \frac{(n-1)! C_{n-1}^{k-1}}{2^{n-k} (k-1)!} + 
\frac{(n-1)! C_{n-1}^{k}}{2^{n-k-1} k!}\cdot \frac{2k+1}{2} 
 +\frac{(n-1)}{4} \cdot 
\frac{(n-1)! C_{n-2}^{k}}{2^{n-k-2} k!}     $$
$$= \frac{(n-1)!}{k! 2^{n-k}}\left(k C_{n-1}^{k-1} +C_{n-1}^{k}(2k+1) +
(n-1) C_{n-2}^{k} \right)       $$
$$ = \frac{(n-1)!}{k! 2^{n-k}}\left(\frac{k(n-1)!}{(k-1)!(n-k)!} +\frac{(2k+1)(n-1)!}{k!(n-1-k)!} 
+ \frac{(n-1)!}{k!(n-2-k)!}   \right)    $$
$$    = \frac{(n-1)!}{k! 2^{n-k}} \cdot \frac{(n-1)!}{k!(n-k)!} \cdot \left(k^2 +(2k+1)(n-k) +(n-1-k)(n-k)\right)$$
$$ = \frac{(n-1)!}{k! 2^{n-k}} \cdot \frac{(n-1)!}{k!(n-k)!} \cdot n^2  =n! \cdot \frac{C_n^k}{2^{n-k}k!}.   $$
Substituting into (\ref{f75}) we get
$$ x S_{n-1}(x) + \frac{(n-1)^2}{4} S_{n-2}(x) =  \frac{n!}{2^n}+ \sum_{k=1}^{n-1}  n! \cdot \frac{C_n^k}{2^{n-k}k!}
\prod_{j=1}^k \left(x- \frac{2j-1}{2}\right)$$
$$ +  n! \cdot \frac{C_n^n}{2^{n-n}n!}\prod_{j=1}^n \left(x- \frac{2j-1}{2}\right)  = S_n(x). $$
Lemma~\ref{L1} is proved.
\end{proof}

Let us use a change of variables to transition from the imaginary axis to the
real axis. We define a new sequence of polynomials
$$M_n(x) = i^{-n} S_n(ix).$$
Using (\ref{f72}) we get that $M_n$ is a real polinomial. From (\ref{f74}) we obtain
$$S_n(ix) =  ix S_{n-1}(ix) + \frac{(n-1)^2}{4} S_{n-2}(ix),   $$
or
\begin{equation}
\label{f76}
M_n(x) = x M_{n-1}(x) -\frac{(n-1)^2}{4} M_{n-2}(x).
\end{equation}
We also know that
$$M_0(x)=1,   M_1(x) =x.$$
Further we will use the famous Favard's theorem.

{\bf Theorem D} (J.~Favard, \cite[pp. 15–16]{RaSc}).  {\it     
Let $(K_n(x))_{n=0}^\infty \subset \mathbb{R}[x],  \deg K_n =n,$ be a sequence of
real polynomials satisfying a recurrence relation of the form
$$K_n(x)=(x-a_n) K_{n-1}(x) -c_n K_{n-2}(x),     $$
where $K_0 =1$ and $c_n >0$ for all $n.$   Then the sequence $(K_n(x))_{n=0}^\infty$ 
forms a sequence of orthogonal polynomials with respect to some positive measure on the 
real line.}

Using (\ref{f76}) and Favard's theorem we conclude that the sequence $(M_n(x))_{n=0}^\infty$ 
forms a sequence of orthogonal polynomials with respect to some positive measure on the 
real line. 

A fundamental property of orthogonal polynomial sequences is that all their roots are real,
simple and that zeros of consecutive polynomials interlace. Thus, all zeros of the polynomials $M_n$ 
are real and simple,  and zeros of consecutive polynomials interlace. Whence, all zeros of the 
polynomials $S_n$ are simple and purely imaginary,  and zeros of consecutive polynomials interlace. 
Equivalently, all zeros of the polynomials  $P_n$ are simple and have real parts equal 
$- \frac{1}{2},$  and zeros of consecutive polynomials interlace.

Theorem~\ref{th1.0} is proved.
\end{proof}

\end{Example}

\section{Open Problems}

In connection with the examples considered we formulate the following open problems.

\begin{Problem}
\label{pr4}
Let
\begin{equation}
\label{f13}
U_n = \{P\in \mathbb{R}[x]  :  \deg P =n\  \mbox{and} \   
(P(k))_{k=0}^\infty \in \mathrm{TP}  \}. 
\end{equation}

Denote by $r_n, n\in \mathbb{N},$ the following numbers
\begin{equation}
\label{f13}
r_n = \sup_{P\in U_n} (\min\{|z| \in \mathbb{C} : P(z)=0 \}).
\end{equation}
The problem is to describe (or estimate) the set of numbers 
$\{r_n\}_{n=1}^\infty.$ Obviously, $r_1=1.$ Example~\ref{ex2} 
shows that $r_2 \geq \frac{1+ \sqrt{2}}{2}>1.$ Example~\ref{ex4} and 
statement 3 of Theorem~\ref{th1} show that $r_{2m+1}=1, m \in \mathbb{N}.$ 
Our conjecture is that the sequence $\{r_n\}_{n=1}^\infty$ is bounded from above.
\end{Problem}

\begin{Problem}
\label{pr5}
Denote by $R_n, n\in \mathbb{N},$ the following numbers
\begin{equation}
\label{f13}
R_n = \sup_{P\in U_n} (\max\{|z| \in \mathbb{C} : P(z)=0 \}).
\end{equation}
The problem is to describe (or estimate) the set of numbers 
$\{R_n\}_{n=1}^\infty.$ Obviously, $R_1=r_1=1.$ Example~\ref{ex4} 
shows that $R_n \geq n.$ We note that item 1 of the Statement~\ref{st1} and 
formula~(\ref{fd2}) show that all real zeros of a real polynomial of degree 
$n$ that interpolated a $\mathrm{TP}$-sequence are situated in the segment $[-n, n-1].$   
Our conjecture is that $R_n = n.$ 
\end{Problem}

\begin{Problem}
\label{pr6}
Let $P = C(x-x_1)(x-x_2) \cdot \ldots \cdot (x-x_n)\in \mathbb{R}[x],$  
$x_1 \leq x_2 \leq \ldots \leq x_n,$  be a polynomial such that  $(P(k))_{k=0}^\infty \in \mathrm{TP}.$
Is it true that there exists a constant $d\geq 1$ (not depending on $n$) such that for every 
$j=1, 2, \ldots, n-1$ we have $x_{j+1} - x_j \leq d?$   
We note that for every polynomial $P$ satisfying the sufficient conditions of 
Theorem~B the statement is true with $d=1.$ Our conjecture is that $d<2.$
\end{Problem}

\end{document}